\documentclass{amsart}
\usepackage{amsfonts, amsbsy, amsmath, amssymb}
\usepackage{xcolor}

\makeatletter
\@namedef{subjclassname@2010}{%
  \textup{2020} Mathematics Subject Classification}
\makeatother

\ifx\fiverm\undefined 
  \newfont\fiverm{cmr5} 
\fi 

\input prepictex.tex
\input pictex.tex
\input postpictex.tex

\newtheorem{thm}{Theorem}[section]
\newtheorem{lem}[thm]{Lemma}
\newtheorem{cor}[thm]{Corollary}

\newtheorem{exmp}[thm]{Example}

\newtheorem{rmk}[thm]{Remark}

\newtheorem{thm-con}[thm]{Theorem-Conjecture}
\numberwithin{equation}{section}

\theoremstyle{definition}
\newtheorem{defn}[thm]{Definition}

\allowdisplaybreaks

\newcommand{\f}{\Bbb F}
\newcommand{\ii}{{\boldsymbol i}}
\newcommand{\jj}{{\boldsymbol j}}
\newcommand{\bt}{{\boldsymbol t}}
\newcommand{\btau}{{\boldsymbol \tau}}

\begin{document}

\title{Some Algebraic Questions about the Reed-Muller Code}

\author[Xiang-dong Hou]{Xiang-dong Hou}
\address{Department of Mathematics and Statistics,
University of South Florida, Tampa, FL 33620}
\email{xhou@usf.edu}



\begin{abstract}
Let $R_q(r,n)$ denote the $r$th order Reed-Muller code of length $q^n$ over $\Bbb F_q$. We consider two algebraic questions about the Reed-Muller code. Let $H_q(r,n)=R_q(r,n)/R_q(r-1,n)$. (1) When $q=2$, it is known that there is a ``duality'' between the actions of $\text{GL}(n,\Bbb F_2)$ on $H_2(r,n)$ and on $H_2(r',n)$, where $r+r'=n$. The result is false for a general $q$. However, we find that a slightly modified duality statement still holds when $q$ is a prime or $r<\text{char}\,\Bbb F_q$. (2) Let $\mathcal F(\Bbb F_q^n,\Bbb F_q)$ denote the $\Bbb F_q$-algebra of all functions from $\Bbb F_q^n$ to $\Bbb F_q$. It is known that when $q$ is a prime, the Reed-Muller codes $\{0\}=R_q(-1,n)\subset R_q(0,n)\subset\cdots\subset R_q(n(q-1),n)=\mathcal F(\Bbb F_q^n,\Bbb F_q)$ are the only $\text{AGL}(n,\Bbb F_q)$-submodules of $\mathcal F(\Bbb F_q^n,\Bbb F_q)$. In particular, $H_q(r,n)$ is an irreducible $\text{GL}(n,\Bbb F_q)$-module when $q$ is a prime. For a general $q$, $H_q(r,n)$ is not necessarily irreducible. We determine all its submodules and the factors in its composition series. The factors of the composition series of $H_q(r,n)$ provide an explicit family of irreducible representations of $\text{GL}(n,\Bbb F_q)$ over $\Bbb F_q$.
\end{abstract}

\maketitle

\section{Introduction}

Let $\mathcal F(\f_q^n,\f_q)$ denote the $\f_q$-algebra of all functions from $\f_q^n$ to $\f_q$. Each such function is uniquely represented by a polynomial $f\in\f_q[X_1,\dots,X_n]$ with $\deg_{X_i}f\le q-1$ for all $1\le i\le n$; polynomials of this form are called {\em reduced}. Each polynomial in $\f_q[X_1,\dots,X_n]$ is congruent to a reduced polynomial modulo the ideal $(X_1^q-X_1,\dots,X_n^q-X_n)$. For $0\le r\le n(q-1)$, the $r$th order {\em Reed-Muller code} of length $q^n$ over $\f_q$ is defined to be
\[
R_q(r,n)=\{f\in\mathcal F(\f_q^n,\f_q):\deg f\le r\}.
\]
In addition, we define $R_q(-1,n)=\{0\}$. There is a natural identification of $\mathcal F(\f_q^n,\f_q)$ with $\f_q^{q^n}$: Each $f\in\mathcal F(\f_q^n,\f_q)$ is identified with its vector of values $(f(x))_{x\in\f_q^n}\in\f_q^{q^n}$. Therefore, $R_q(r,n)$ is an $\f_q$-subspace of $\f_q^{q^n}$ and hence is a linear code of length $q^n$ over $\f_q$; this is the context in which the Reed-Muller code was initially discovered with $q=2$.

The {\em affine linear group} of degree $n$ over $\f_q$ is
\[
\text{AGL}(n,\f_q)=\Bigl\{\left[\begin{matrix} A&0\cr a&1\end{matrix}\right]:A\in\text{GL}(n,\f_q),\ a\in\f_q^n\Bigr\}<\text{GL}(n+1,\f_q).
\] 
The group $\text{AGL}(n,\f_q)$ acts on $\mathcal F(\f_q^n,\f_q)$ as follows: For $\sigma=\left[\begin{smallmatrix}A&0\cr a&1\end{smallmatrix}\right]\in \text{AGL}(n,\f_q)$ and $f\in\mathcal F(\f_q^n,\f_q)$,
\[
\sigma(f)=f((X_1,\dots,X_n)A+a),
\] 
that is, $\sigma(f)=f\circ\sigma$, where the $\sigma$ in $f\circ\sigma$ is treated as an affine transformation of $\f_q^n$.
Under this action, the Reed-Muller codes $R_q(r,n)$ become $\text{AGL}(n,\f_q)$-modules. (In fact, except for the extreme cases $r\in\{-1,0,n(q-1)-1,n(q-1)\}$, $\text{AGL}(n,\f_q)$ is the largest subgroup $G$ of the permutation group on $\f_q^n$ such that $R_q(r,n)$ is $G$-invariant; in coding theoretic terms, $\text{AGL}(n,\f_q)$ is the automorphism group of $R_q(r,n)$ except for the extreme cases \cite{Berger-Charpin-DM-1993}.) Interesting algebraic questions arise about these $\text{AGL}(n,\f_q)$-modules. The quotient module 
\begin{equation}\label{Hq}
H_q(r,n):=R_q(r,n)/R_q(r-1,n)
\end{equation} 
consists of reduced homogeneous polynomials of degree $r$ in $\f_q[X_1,\dots,X_n]$. Translations $(X_1,\dots,X_n)\mapsto(X_1+a_1,\dots,X_n+a_n)$ have no effect on $H_q(r,n)$. Hence, the $\text{AGL}(n,\f_q)$-structure of $H_q(r,n)$ induces a $\text{GL}(n,\f_q)$-module structure. In this paper, we consider two separate questions about the module $H_q(r,n)$.

Let $\Omega_{q,n}=\{0,1,\dots,q-1\}^n$. For $\ii=(i_1,\dots,i_n)\in \Omega_{q,n}$, define $|\ii|=i_1+\cdots+i_n$, $\bar\ii=(q-1-i_1,\dots,q-1-i_n)$, $X^{\ii}=X_1^{i_1}\cdots X_n^{i_n}\in\mathcal F(\f_q^n,\f_q)$, and, for $0\le r\le n(q-1)$, define $\Omega_{q,n,r}=\{\ii\in \Omega_{q,n}:|\ii|=r\}$. Then $H_q(r,n)$ has an $\f_q$-basis $\{X^\ii:\ii\in \Omega_{q,n,r}\}$, and the ``dual'' module $H_q(r',n)$, where $r+r'=n(q-1)$, has a ``dual'' basis $\{(-1)^nX^{\bar \ii}:\ii\in \Omega_{q,n,r}\}$. Let $(\ )^c:H_q(r,n)\to H_q(r',n)$ be the $\f_q$-map sending $X^{\ii}$ to $(-1)^nX^{\bar\ii}$. When $q=2$, it is known that $f,g\in H_2(r,n)$ are GL-equivalent (i.e., in the same $\text{GL}(n,\f_2)$-orbit) if and only if $f^c,g^c\in H_2(r',n)$ are GL-equivalent \cite[\S4]{Hou-DM-1996}. For a general $q$, this duality statement is not true; see Example~\ref{E2.4}. However, we will prove in Theorem~\ref{T2.2} that a slightly modified duality statement still holds when $q$ is a prime or $r<\text{char}\,\f_q$.

When $q=p$ is a prime, Mortimer \cite[Ch.~5]{Mortimer-thesis-1977} proved that the Reed-Muller codes
\[
\{0\}=R_p(-1,n)\subset R_p(0,n)\subset \cdots\subset R_p(n(p-1),n)=\mathcal F(\f_p^n,\f_p)
\]
are the only $\text{AGL}(n,\f_p)$-submodules of $\mathcal F(\f_p^n,\f_p)$; also see \cite[\S 5.5]{Assmus-Key-1998}. In particular, $H_p(r,n)$ is an irreducible $\text{GL}(n,\f_p)$-module. However, for a general $q$, $H_q(r,n)$ is not necessarily irreducible and its submodules have not been determined previously. Our second main result (Theorem~\ref{T3.8}) gives all $\text{GL}(n,\f_q)$-submodules of $H_q(r,n)$. Moreover, we determine the factors in the composition series of $H_q(r,n)$. Consequently, we obtain a class of irreducible modular representations of $\text{GL}(n,\f_q)$ over $\f_q$. There is a method for constructing all irreducible $\f_q\text{GL}(n,\f_q)$-modules using Weyl modules. The factors of the composition series of $H_q(r,n)$, though accounting for a small portion of all irreducible $\f_q\text{GL}(n,\f_q)$-modules, have the advantage that they are explicit and much easier to describe. More comments in this regard are given in Section~5.

The above questions and their solutions have practical applications in coding theory. The duality between $H_q(r,n)$ and $H_q(r',n)$, when it exists, allows people to study the homogeneous $q$-ary functions of degree $r'$ through the canonical homogeneous $q$-ary functions of degree $r$. The duality between $H_2(2,n)$ and $H_2(n-2,n)$ played an essential role in a simplified approach to the determination of the covering radius of $R_2(1,7)$ \cite{Hou-JCTA-1996, Mykkeltveit-IEEE-IT-1980}.  When $q=2$, the canonical forms in $H_2(3,n)$ are known for $n\le 9$ \cite{Brier-Langevin-ITW-2003, Brier-Langevin-web, Hou-DM-1996}. (Elements of $H_2(3,n)$ are reduced binary cubic forms in $n$ variables.) Recently, these results and the duality between $H_2(3,n)$ and $H_2(n-3,n)$ were used by Dougherty, Mauldin and Tiefenbruck \cite{Dougherty-Mauldin-Tiefenbruck-arXiv:2106.13910} to study the covering radius of $R_2(n-4,n)$ in $R_2(n-3,n)$. The $\text{GL}(n,\f_q)$-submodules of $H_q(r,n)$ correspond to the $\text{AGL}(n,\f_q)$-submodules between $R_q(r-1,n)$ and $R_q(r,n)$. $\text{AGL}(n,\f_q)$-submodules of $\mathcal F(\f_q^n,\f_q)$ are codes whose automorphism groups contain $\text{AGL}(n,\f_q)$; they belong to the class of affine invariant codes. For studies on other types of affine invariant codes, see \cite{Berger-Charpin-IEEE-IT-1996, Charpin-Levy-Dit-Vehel-JCTA-1994, Delsarte-IEEE-IT-1970, Hou-JCTA-2005, Hou-IJICT-2010, Kasami-Lin-Peterson-IC-1968}.
 In general, codes with large automorphism groups facilitate effective decoding schemes such as permutation decoding \cite{Key-McDonough-Mavron-DM-2010, Key-McDonough-Mavron-DM-2017, MacWilliams-BSTJ}.

To simplify writing, we will allow a few harmless abuses of notation. When a polynomial $f\in\f_q[X_1,\dots,X_n]$ is treated as an element of $\mathcal F(\f_q^n,\f_q)$, it is meant to be the reduced polynomial $f_1$ such that $f\equiv f_1\pmod{(X_1^q-X_1,\dots,X_n^q-X_n)}$. When a polynomial $f\in\f_q[X_1,\dots,X_n]$ of degree $\le r$ is treated as an element of $H_q(r,n)$, it is meant to be the coset $f_1+R_q(r-1,n)$, where $f_1$ is the reduced polynomial of $f$.

\section{A Duality Theorem}

Define an inner product $\langle\,\cdot\, ,\,\cdot\,\rangle$ on $\mathcal F(\f_q^n,\f_q)$ by
\begin{equation}\label{2.1}
\langle f,g\rangle=\sum_{x\in\f_q^n}f(x)g(x),\qquad f,g\in\mathcal F(\f_q^n,\f_q).
\end{equation}
For $\sigma\in\text{AGL}(n,\f_q)$ and $f,g\in\mathcal F(\f_q^n,\f_q)$, we have
\begin{equation}\label{adj} 
\langle\sigma(f),g\rangle=\sum_{x\in\f_q^n}f(\sigma(x))g(x)=\sum_{y\in\f_q^n}f(y)g(\sigma^{-1}(y))=\langle f,\sigma^{-1}(g)\rangle.
\end{equation}
Hence, when $\sigma$ is treated as an $\f_q$-linear transformation of $\mathcal F(\f_q^n,\f_q)$, its adjoint with respect to the inner product $\langle\,\cdot\, ,\,\cdot\,\rangle$ is $\sigma^{-1}$.
For $0\le r,r'\le n(q-1)$ with $r+r'=n(q-1)$, the above inner product induces a well-defined non-degenerate $\f_q$-bilinear map $\langle\,\cdot\, ,\,\cdot\,\rangle:H_q(r,n)\times H_q(r',n)\to \f_q$. In fact, if $f_1,f_2\in R_q(r,n)$ and $g_1,g_2\in R_q(r',n)$ are such that $f_1\equiv f_2\pmod{R_q(r-1,n)}$ and $g_1\equiv g_2\pmod{R_q(r'-1,n)}$, then it is easy to see that
\[
\langle f_1,g_1\rangle=\langle f_2,g_2\rangle.
\]
The map $\langle\,\cdot\, ,\,\cdot\,\rangle:H_q(r,n)\times H_q(r',n)\to \f_q$ will be referred to as the {\em pairing} between $H_q(r,n)$ and $H_q(r',n)$. The module $H_q(r,n)$ has an $\f_q$-basis 
\[
\frak B_r=\{X^{\ii}:\ii\in \Omega_{q,n,r}\},
\]
which is ordered by the lexicographic order on $\Omega_{q,n,r}$. 
The dual basis of $\frak B_r$ in $H_q(r',n)$ with respect to the pairing $\langle\,\cdot\, ,\,\cdot\,\rangle$ is 
\[
\frak B'_r=\{(-1)^nX^{\bar\ii}:\ii\in \Omega_{q,n,r}\},
\]
as one can easily see that for $\ii,\jj\in\Omega_{q,n,r}$,
\begin{equation}\label{2.2}
\langle X^{\ii},(-1)^nX^{\bar\jj}\rangle=
\begin{cases}
1&\text{if}\ \ii=\jj,\cr
0&\text{if}\ \ii\ne\jj.
\end{cases}
\end{equation}
For $A\in\text{GL}(n,\f_q)$, the action of $A$ on $f\in\mathcal F(\f_q^n,\f_q)$ is
\[
A(f)=f((X_1,\dots,X_n)A).
\]
Let $\mathcal A_r(A)$ denote the matrix of $A$ (as an $\f_q$-linear transformation of $H_q(r,n)$) with respect to the basis $\frak B_r$. Then by \eqref{2.2}, $\mathcal A_r(A)=(\sigma_{\ii,\jj}(A))_{\ii,\jj\in \Omega_{q,n,r}}$, where
\begin{equation}\label{sigmaij}
\sigma_{\ii,\jj}(A)=\langle A(X^{\ii}),(-1)^nX^{\bar\jj}\rangle.
\end{equation}
Let $A=(a_{ts})_{1\le t,s\le n}$ and $\ii=(i_1,\dots,i_n), \jj=(j_1,\dots,j_n)\in \Omega_{q,n,r}$. Then 
\begin{align*}
\sigma_{\ii,\jj}(A)\,&=(-1)^n\sum_{x=(x_1,\dots,x_n)\in\f_q^n}(a_{11}x_1+\cdots+a_{n1}x_n)^{i_1}\cdots(a_{1n}x_1+\cdots+a_{nn}x_n)^{i_n}\cr
&\kern10em \cdot x_1^{q-1-j_1}\cdots x_n^{q-1-j_n}\cr
&=(-1)^n\sum_{x\in\f_q^n}\Bigl(\sum_{i_{11}+\cdots+i_{1n}=i_1}\binom{i_1}{i_{11},\dots,i_{1n}}(a_{11}x_1)^{i_{11}}\cdots(a_{n1}x_n)^{i_{1n}}\Bigr)\cdots\cr
&\kern6em \Bigl(\sum_{i_{n1}+\cdots+i_{nn}=i_n}\binom{i_n}{i_{n1},\dots,i_{nn}}(a_{1n}x_1)^{i_{n1}}\cdots(a_{nn}x_n)^{i_{nn}}\Bigr)\cr
&\kern6em \cdot x_1^{q-1-j_1}\cdots x_n^{q-1-j_n},
\end{align*}
where
\[
\binom{i_s}{i_{s1},\dots,i_{sn}}=\frac{i_s!}{i_{s1}!\cdots i_{sn}!}
\]
is the multinomial coefficient. Therefore,
\begin{align*}
\sigma_{\ii,\jj}(A)\,
&=(-1)^n\sum_{\substack{(i_{st})\cr \sum_ti_{st}=i_s,\,1\le s\le n}}\binom{i_1}{i_{11},\dots,i_{1n}}\cdots\binom{i_n}{i_{n1},\dots,i_{nn}}\Bigl(\prod_{s,t}a_{ts}^{i_{st}}\Bigr)\cr
&\kern5em \cdot\sum_{x\in\f_q^n}x_1^{i_{11}+\cdots+i_{n1}+q-1-j_1}\cdots x_n^{i_{1n}+\cdots+i_{nn}+q-1-j_n}.
\end{align*}
In the above, 
\begin{align*}
&\sum_{x\in\f_q^n}x_1^{i_{11}+\cdots+i_{n1}+q-1-j_1}\cdots x_n^{i_{1n}+\cdots+i_{nn}+q-1-j_n}\cr
&=\begin{cases}
(-1)^n&\text{if\ $\displaystyle\sum_s i_{st}+q-1-j_t$ is $>0$\ and\ $\equiv 0\ \text{mod}\;(q-1)$ for all $1\le t\le n$},\cr
0&\text{otherwise}.
\end{cases}
\end{align*}
Hence
\[
\sigma_{\ii,\jj}(A)
=\sum_{(i_{st})}\binom{i_1}{i_{11},\dots,i_{1n}}\cdots\binom{i_n}{i_{n1},\dots,i_{nn}}\Bigl(\prod_{s,t}a_{ts}^{i_{st}}\Bigr),
\]
where the sum is over all the matrices $(i_{st})$ subject to the conditions
\begin{equation}\label{cond-ist}
\begin{cases}
i_{st}\ge 0\ \text{for all}\ 1\le s,t\le n,\vspace{0.4em}\cr
\displaystyle\sum_t i_{st}=i_s\ \text{for all}\ 1\le s\le n,\vspace{0.4em}\cr
\displaystyle\sum_s i_{st}\equiv j_t\ \text{mod}\;(q-1)\ \text{and}\ >j_t-q+1\ \text{for all}\ 1\le t\le n.
\end{cases}
\end{equation}
The first and third conditions in \eqref{cond-ist} imply that $\sum_s i_{st}\ge j_t$ for all $1\le t\le n$. (Look at the cases $0\le j_t<q-1$ and $j_t=q-1$ separately.) 
Since $|\ii|=|\jj|$, the conditions $\sum_ti_{st}=i_s$ (for all $s$) and $\sum_si_{st}\ge j_t$ (for all $t$) imply that $\sum_si_{st}=j_t$ for all $t$. Let 
\[
M(\ii,\jj)=\Bigl\{(i_{st})_{1\le s,t\le n}:\sum_ti_{st}=i_s,\ 1\le s\le n;\ \sum_si_{st}=j_t,\ 1\le t\le n\Bigr\};
\]
see Figure~\ref{F1}. Then 
\begin{equation}\label{sigma}
\sigma_{\ii,\jj}(A)=\sum_{(i_{st})\in M(\ii,\jj)}\binom{i_1}{i_{11},\dots,i_{1n}}\cdots\binom{i_n}{i_{n1},\dots,i_{nn}}\Bigl(\prod_{s,t}a_{ts}^{i_{st}}\Bigr).
\end{equation}
For $\ii=(i_1,\dots,i_n)\in \Omega_{q,n,r}$, define $\ii!:=i_1!\cdots i_n!$. Let $D_r$ be the $\Omega_{q,n,r}\times \Omega_{q,n,r}$ diagonal matrix whose $(\ii,\ii)$ entry is $\ii!$.

\begin{figure}
\[
\beginpicture
\setcoordinatesystem units <1.4em,1.4em> point at 0 0

\arrow <5pt> [.2,.67] from 7.5 0 to 8.5 0 
\arrow <5pt> [.2,.67] from 7.5 6 to 8.5 6
\arrow <5pt> [.2,.67] from 0 -1.5 to 0 -2.5
\arrow <5pt> [.2,.67] from 6 -1.5 to 6 -2.5  
\setlinear
\plot -1 -1  7 -1  7 7  -1 7  -1 -1 /

\put {$i_{n1}$} at 0 0
\put {$\cdot$} at 2 0
\put {$\cdot$} at 3 0
\put {$\cdot$} at 4 0
\put {$i_{nn}$} at 6 0
\put {$i_{11}$} at 0 6
\put {$\cdot$} at 2 6
\put {$\cdot$} at 3 6
\put {$\cdot$} at 4 6
\put {$i_{1n}$} at 6 6
\put {$\cdot$} at 0 2
\put {$\cdot$} at 0 3
\put {$\cdot$} at 0 4
\put {$\cdot$} at 6 2
\put {$\cdot$} at 6 3
\put {$\cdot$} at 6 4
\put {$i_n$} at 9 0
\put {$i_1$} at 9 6
\put {$j_1$} at 0 -3
\put {$j_n$} at 6 -3
\endpicture
\]
\caption{Matrices in $M(\ii,\jj)$}\label{F1}
\end{figure}

\begin{lem}\label{L2.1}
For $A\in\text{\rm GL}(n,\f_q)$, we have
\begin{equation}\label{AD=DA}
\mathcal A_r(A^T)D_r=D_r\mathcal A_r(A)^T.
\end{equation}
\end{lem}

\begin{proof}
Let $A=(a_{ts})_{1\le t,s\le n}$. We treat $a_{ts}$ as independent indeterminates and thus we only have to prove \eqref{AD=DA} over the ring $\Bbb Z[\{a_{ts}:1\le t,s\le n\}]$. Therefore, we only have to prove \eqref{AD=DA} over the ring $\Bbb Q[\{a_{ts}:1\le t,s\le n\}]$. 

For $\ii=(i_1,\dots,i_n), \jj=(j_1,\dots,j_n)\in \Omega_{q,n,r}$, by \eqref{sigma}, the $(\jj,\ii)$ entry of $\mathcal A_r(A^T)D_r$ is
\begin{align*}
\sigma_{\jj,\ii}(A^T)\cdot\ii!\,&=\ii! \sum_{(j_{st})\in M(\jj,\ii)}\binom{j_1}{j_{11},\dots,j_{1n}}\cdots\binom{j_n}{j_{n1},\dots,j_{nn}}\Bigl(\prod_{s,t}a_{st}^{j_{st}}\Bigr)\cr
&=\ii!\,\jj!\sum_{(j_{st})\in M(\jj,\ii)}\prod_{s,t}\frac{a_{st}^{j_{st}}}{j_{st}!}\cr
&=\ii!\,\jj!\sum_{(j_{ts})\in M(\ii,\jj)}\prod_{s,t}\frac{a_{ts}^{j_{ts}}}{j_{ts}!}\cr
&=\ii!\,\jj!\sum_{(i_{st})\in M(\ii,\jj)}\prod_{s,t}\frac{a_{ts}^{i_{st}}}{i_{st}!}\kern 9.5em (i_{st}=j_{ts})\cr
&=\jj! \sum_{(i_{st})\in M(\ii,\jj)}\binom{i_1}{i_{11},\dots,i_{1n}}\cdots\binom{i_n}{i_{n1},\dots,i_{nn}}\Bigl(\prod_{s,t}a_{ts}^{i_{st}}\Bigr)\cr
&=\jj!\cdot\sigma_{\ii,\jj}(A),
\end{align*}
which is the $(\ii,\jj)$ entry of $\mathcal A_r(A)D_r$. Hence
\[
\mathcal A_r(A^T)D_r=(\mathcal A_r(A)D_r)^T=D_r\mathcal A_r(A)^T.
\]
\end{proof}

For $0\le r,r'\le n(q-1)$ with $r+r'=n(q-1)$, let 
\[
\theta:H_q(r,n)\to H_q(r',n)
\]
be the $\f_q$-linear map sending $X^{\ii}$ to $\ii!\,(-1)^nX^{\bar\ii}$, $\ii\in \Omega_{q,n,r}$.

\begin{thm}\label{T2.2}
Let $0\le r,r'\le n(q-1)$ be such that $r+r'=n(q-1)$. Then for each $A\in\text{\rm GL}(n,\f_q)$, the following diagram commutes.

\[
\beginpicture
\setcoordinatesystem units <2em,2em> point at 0 0

\arrow <5pt> [.2,.67] from 1.5 0 to 4.5 0 
\arrow <5pt> [.2,.67] from 1.5 3 to 4.5 3
\arrow <5pt> [.2,.67] from 0 2.5 to 0 0.5 
\arrow <5pt> [.2,.67] from 6 2.5 to 6 0.5  

\put {$H_q(r',n)$} at 0 0
\put {$H_q(r',n)$} at 6 0
\put {$H_q(r,n)$} at 0 3
\put {$H_q(r,n)$} at 6 3
\put {$\scriptstyle A$} [b] at 3 3.2
\put {$\scriptstyle (A^{-1})^T$} [b] at 3 0.2
\put {$\scriptstyle \theta$} [r] at -0.2 1.5
\put {$\scriptstyle \theta$} [l] at 6.2 1.5

\endpicture
\]
\vspace{0.1em}

\noindent In particular, when $q$ is a prime or $r<\text{\rm char}\,\f_q$, $f,g\in H_q(r,n)$ are GL-equivalent if and only if $\theta(f),\theta(g)\in H_q(r',n)$ are GL-equivalent.
\end{thm}

\begin{proof}
First, the matrix of the $\f_q$-linear map $\theta\circ A:H_q(r,n)\to H_q(r',n)$ with respect to the basis $\frak B_r$ of the domain and the basis $\frak B_r'$ of the target is $\mathcal A_r(A)D_r$.

On the other hand, let $B=(A^{-1})^T$. The matrix of $B^{-1}:H_q(r,n)\to H_q(r,n)$ with respect to the basis $\frak B_r$ is $\mathcal A_r(B^{-1})$. By \eqref{adj}, the adjoint of $B^{-1}:H_q(r,n)\to H_q(r,n)$ is $B:H_q(r',n)\to H_q(r',n)$. Thus the matrix of $B:H_q(r',n)\to H_q(r',n)$ with respect to the basis $\frak B_r'$ is $\mathcal A_r(B^{-1})^T$ \cite[Chapter~XIII, Corollary~7.4]{Lang-2002}.
Hence the matrix of $B\circ\theta:H_q(r,n)\to H_q(r',n)$ with respect to the basis $\frak B_r$ of the domain and the basis $\frak B_r'$ of the target is $D_r(\mathcal A_r(B^{-1}))^T$. Therefore, it remains to verify that $D_r(\mathcal A_r(B^{-1}))^T=\mathcal A_r(A)D_r$. By Lemma~\ref{L2.1}, we have
\[
D_r(\mathcal A_r(B^{-1}))^T=D_r(\mathcal A_r(A^T))^T=\mathcal A_r(A)D_r.
\]
\end{proof}

\begin{rmk}\label{R2.3}\rm
(i) If $q$ is not a prime and $r\ge\text{char}\,\f_q$, then the map $\theta$ in Theorem~\ref{T2.2} is not invertible. Hence the ``if'' part of the second statement in Theorem~\ref{T2.2} is false.

\medskip

(ii) The special case of Theorem~\ref{T2.2} with $q=2$ was first proved in \cite{Hou-DM-1996}. In this case, $\theta=(\ )^c$ and $\mathcal A_r(A)$ is the $r$th {\em compound matrix} of $A$, which is a critical fact that the proof in \cite{Hou-DM-1996} relied on. (For the definition and properties of compound matrices, see \cite[Ch. V]{Wedderburn-1964}.) However, when $q>2$, the connection with compound matrices no longer exists. For this reason, the proof of Theorem~\ref{T2.2} given above is not a simple adaptation of the proof of the special case $q=2$ in \cite{Hou-DM-1996}.

\medskip
(iii) If $q$ is not a prime and $r\ge\text{char}\,\f_q$, unlike $\theta$, $(\ )^c:H_q(r,n)\to H_q(r',n)$ is still invertible. Can we expect the second statement in Theorem~\ref{T2.2} to be true with $\theta$ replaced by $(\ )^c$? The following example gives a negative answer. 
\end{rmk}

\begin{exmp}\label{E2.4}\rm
Let $q=4$, $n=2$, $r=4$, and $f=X_1^3X_2$, $g=X_1^3X_2+X_1^2X_2^2+X_1X_2^3\in H_4(4,2)$. Let $\sim$ denote GL-equivalence. Then
\[
f\sim(X_1+X_2)^3X_2=(X_1^3+X_1^2X_2+X_1X_2^2+X_2^3)X_2=X_1^3X_2+X_1^2X_2^2+X_1X_2^3=g.
\]
However, in $H_4(2,2)$, $f^c=X_2^2$ and $g^c=X_2^2+X_1X_2+X_1^2$, which are not GL-equivalent since $f^c$ is a quadratic form of rank $1$ and $g^c$ is a quadratic form of rank $2$.
\end{exmp}

\section{$\text{GL}(n,\f_q)$-Submodules of $H_q(r,n)$}

Let $1\le r\le n(q-1)$. The objective of this section is to determine all $\text{GL}(n,\f_q)$-submodules of $H_q(r,n)$.  Let $q=p^m$, where $p=\text{\rm char}\,\f_q$. Let $M$ be a nonzero $\text{GL}(n,\f_q)$-module in $H_q(r,n)$.

\begin{lem}\label{L2}
Assume that
\[
X_n^{q-2}a_{q-2}+X_n^{q-3}a_{q-3}+\cdots+a_0\in M,
\]
where $a_i\in H_q(r-i,n-1)$. Then $X_n^ia_i\in M$ for all $0\le i\le q-2$.
\end{lem}

\begin{proof}
Let $f(X_1,\dots,X_n)$ denote the polynomial in the lemma. For all $c\in\f_q^*$, we have 
\[
f(X_1,\dots,X_{n-1},cX_n)=(c^{q-2},c^{q-3},\dots,1)\left[\begin{matrix}
X_n^{q-2}a_{q-2}\cr X_n^{q-3}a_{q-3}\cr \vdots\cr a_0\end{matrix}\right]\in M.
\]
The rows $(c^{q-2},c^{q-3},\dots,1)$, $c\in\f_q^*$, are linearly independent since they are from a Vandermonde matrix. Thus
$X_n^ia_i\in M$ for all $0\le i\le q-2$.
\end{proof}

\begin{lem}\label{L4}
Assume that $X_1^{i_1}X_2^{i_2}\cdots X_n^{i_n}\in M$, where $(i_1,\dots,i_n)\in\Omega_{q,n,r}$, and write 
\[
i_1=c_0p^0+\cdots +c_{m-1}p^{m-1}, 
\]
where $0\le c_j\le p-1$. If $c_k>0$, then $X_1^{i_1-p^k}X_2^{i_2+p^k}\cdots X_n^{i_n}\in M$.
\end{lem}

\begin{proof}
We have
\begin{align*}
M\ni\,&(X_1+X_2)^{i_1}X_2^{i_2}\cdots X_n^{i_n}-X_1^{i_1}X_2^{i_2}\cdots X_n^{i_n}\cr
=\,&\Bigl[\sum_{j=1}^{i_1}\binom{i_1}jX_1^{i_1-j}X_2^j\Bigr]X_2^{i_2}\cdots X_n^{i_n}\cr
=\,&\sum_{j=1}^{i_1}\binom{i_1}jX_1^{i_1-j}X_2^{j+i_2}\cdots X_n^{i_n}.
\end{align*}
By Lemma~\ref{L2}, $\binom{i_1}jX_1^{i_1-j}X_2^{j+i_2}\cdots X_n^{i_n}\in M$ for all $1\le j\le i_1$. Choosing $j=p^k$ gives $X_1^{i_1-p^k}X_2^{i_2+p^k}\cdots X_n^{i_n}\in M$. (Note that $\binom{i_1}{p^k}\ne 0$.)
\end{proof}

\begin{lem}\label{L8}
If $f\in M$, then every monomial in $f$ belongs to $M$.
\end{lem}

\begin{proof}
Use induction on $n$.

First we claim that if $X_n^ia(X_1,\dots,X_{n-1})\in M$, where $a\in H_q(r-i,n-1)$, then all monomials in $X_n^ia$ are in $M$. Let $M_1=\{b\in H_q(r-i,n-1):X_n^ib\in M\}$. Then $M_1$ is a $\text{GL}(n-1,\f_q)$-module and $a\in M_1$. By the induction hypothesis, all monomials in $a$ are in $M_1$. Hence all monomials in $X_n^ia$ are in $M$.

Let
\[
f=X_n^{q-1}a_{q-1}(X_1,\dots,X_{n-1})+X_n^{q-2}a_{q-2}(X_1,\dots,X_{n-1})+\cdots+
a_{0}(X_1,\dots,X_{n-1}).
\]
By the above claim, it suffices to show that $X_n^ia_i\in M$ for all $0\le i\le q-1$. Let $\gamma$ be a primitive element of $\f_q$. Then
\[
f(X_1,\dots,X_{n-1},\gamma X_n)-f(X_1,\dots,X_{n-1}, X_n)=\sum_{i=1}^{q-2}(\gamma^i-1)X_n^ia_i\in M.
\]
By Lemma~\ref{L2}, $X_n^ia_i\in M$ for all $1\le i\le q-2$. Thus we also have $X_n^{q-1}a_{q-1}+a_0\in M$. Let
\[
M_2=\{b\in H_q(r,n-1):X_n^{q-1}a+b\in M\ \text{for some}\ a\in H_q(r-(q-1),n-1)\}.
\]
Then $M_2$ is a $\text{GL}(n-1,\f_q)$-module and $a_0\in M_2$. By the induction hypothesis, all monomials of $a_0$ are in $M_2$, that is, for any monomial $X_1^{i_1}\cdots X_{n-1}^{i_{n-1}}$ of $a_0$, there exists $a(X_1,\dots,X_{n-1})\in H_q(r-(q-1),n-1)$ such that $X_n^{q-1}a+X_1^{i_1}\cdots X_{n-1}^{i_{n-1}}\in M$. It suffices to show that $X_1^{i_1}\cdots X_{n-1}^{i_{n-1}}\in M$. (Then $a_0\in M$ and hence $X_n^{q-1}a_{q-1}\in M$.) Without loss of generality, assume $i_1>0$. Write
\[
i_1=c_0p^0+\cdots+c_{m-1}p^{m-1}
\]
in base $p$ expansion and assume that $c_j>0$ for some $j$. Then
\begin{align*}
M\ni\,&\Bigl(X_n^{q-1}a_{q-1}((X_1+X_n),X_2,\dots,X_{n-1})+(X_1+X_n)^{i_1}X_2^{i_2}\cdots X_{n-1}^{i_{n-1}} \Bigr)\cr\
&-\Bigl(X_n^{q-1}a_{q-1}(X_1,X_2,\dots,X_{n-1})+X_1^{i_1}X_2^{i_2}\cdots X_{n-1}^{i_{n-1}} \Bigr)\cr
=\,& (X_1+X_n)^{i_1}X_2^{i_2}\cdots X_{n-1}^{i_{n-1}}-X_1^{i_1}X_2^{i_2}\cdots X_{n-1}^{i_{n-1}} \cr
=\,&\sum_{k=1}^{i_1}\binom{i_1}k X_n^kX_1^{i_1-k}X_2^{i_2}\cdots X_{n-1}^{i_{n-1}}.
\end{align*}
(To see the first equality in the above, note that $X_n^q=X_n$ in $\mathcal F(\f_q^n,\f_q)$, and hence any monomial $X_1^{j_1}\cdots X_n^{j_n}$ with $j_1+\cdots+j_n=r$ and $j_n\ge q$ is 0 in $H_q(r,n)$.)
By Lemma~\ref{L2}, $\binom{i_1}k X_n^kX_1^{i_1-k}X_2^{i_2}\cdots X_{n-1}^{i_{n-1}}\in M$ for all $1\le k\le i_1$. (Note: $X_1$ here plays the role of $X_n$ in Lemma~\ref{L2}.) Since $c_j>0$, by Lucas's theorem, $\binom{i_1}{p^j}\ne0$, whence $X_n^{p^j}X_1^{i_1-p^j}X_2^{i_2}\cdots X_{n-1}^{i_{n-1}}\in M$. Then by Lemma~\ref{L4}, $X_1^{i_1}X_2^{i_2}\cdots X_{n-1}^{i_{n-1}}\in M$. This completes the proof.
\end{proof}

\begin{cor}\label{C3.4}
$M$ is generated  over $\f_q$ by a set of monomials.
\end{cor}

\begin{proof}
This is a restatement of Lemma~\ref{L8}.
\end{proof}

For $\ii=(i_1,\dots,i_n)\in \Omega_{q,n,r}$, write
\[
i_j=\sum_{k=0}^{m-1} i_{jk}p^k,\quad 0\le i_{jk}\le p-1,
\]
and let
\[
D(\ii)=
\left[\begin{matrix}i_{10}&\cdots&i_{1,m-1}\cr\vdots&&\vdots\cr i_{n0}&\cdots&i_{n,m-1}\end{matrix}\right].
\]
Define $T(\ii)\in \Bbb N^m$ by
\begin{align}\label{Ti}
T(\ii)\,&=\Bigl(\sum_{k=0}^0\sum_{j=1}^n i_{jk}p^k,\, \sum_{k=0}^1\sum_{j=1}^n i_{jk}p^k,\, \dots,\, \sum_{k=0}^{m-1}\sum_{j=1}^n i_{jk}p^k\Bigr)\\
&=\left[\begin{matrix}1&\cdots&1\end{matrix}\right]
D(\ii)
\left[\begin{matrix}p^0&p^0&\cdots&p^0\cr &p^1&\cdots&p^1\cr &&\ddots&\vdots\cr &&&p^{m-1}\end{matrix}\right].\nonumber
\end{align}
In the above, $T(\ii)$ is determined by $\left[\begin{matrix}1&\cdots&1\end{matrix}\right]
D(\ii)$, the vector of column sums of $D(\ii)$, and vice versa. 
Equation~\eqref{Ti} defines a map $T:\Omega_{q,n,r}\to\Bbb N^m$. 
Let $T(\ii)=(t_0,\dots,t_{m-1})$. Then 
\begin{equation}\label{tk+=r}
t_k+p^{k+1}\Bigl(\Bigl\lfloor\frac{i_1}{p^{k+1}}\Bigr\rfloor+\cdots+\Bigl\lfloor\frac{i_n}{p^{k+1}}\Bigr\rfloor\Bigr)=r,\quad 0\le k\le m-1,
\end{equation}
and hence $t_{m-1}=r$ and 
\[
t_k\equiv r\pmod{p^{k+1}},\quad 0\le k\le m-2.
\]
The image set $T(\Omega_{q,n,r})$ consists of $m$-tuples $(t_0,\dots,t_{m-1})$ satisfying the following conditions:
\begin{equation}\label{TOmega}
\begin{cases}
t_{m-1}=r,\cr
t_k\equiv r\pmod{p^{k+1}},\quad 0\le k\le m-2,\cr
0\le t_0\le n(p-1),\vspace{0.2em}\cr
0\le\displaystyle\frac 1{p^k}(t_k-t_{k-1})\le n(p-1),\quad 1\le k\le m-1.
\end{cases}
\end{equation}
For $\bt=(t_0,\dots,t_{m-1}), \bt'=(t_0',\dots,t_{m-1}')\in T(\Omega_{q,n,r})$, define $\bt\le \bt'$ if $t_j\le t_j'$ for all $0\le j\le m-1$. Then $(T(\Omega_{q,n,r}),\le)$ is a partially ordered set.

\begin{lem}\label{L3.5}
Let $\ii\in\Omega_{q,n,r}$ and $A\in\text{\rm GL}(n,\f_q)$. Then in $H_q(r,n)$,
\begin{equation}\label{AXi}
A(X^\ii)=\sum_{\substack{\jj\in \Omega_{q,n,r}\cr T(\jj)\le T(\ii)}}\alpha_{\jj}X^{\jj},
\end{equation}
where $\alpha_\jj\in\f_q$.
\end{lem}

\begin{proof}
Let $\ii=(i_1,\dots,i_n)$. If $(X_1,\dots,X_n)A=(\lambda X_1,X_2,\dots,X_n)$, where $\lambda\in\f_q^*$, then $A(X^\ii)=\lambda^{i_1} X^\ii$. If $(X_1,\dots,X_n)A=(X_2,X_1,X_3,\dots,X_n)$, then $A(X^\ii)=X^\jj$, where $\jj=(i_2,i_1,i_3,\dots,i_n)\in\Omega_{q,n,r}$ and $T(\jj)=T(\ii)$. It remains to consider the case $(X_1,\dots,X_n)A=(X_1+X_2,X_2,\dots,X_n)$. We have
\begin{align*}
A(X^\ii)\,&=(X_1+X_2)^{i_1}X_2^{i_2}\cdots X_n^{i_n}=\sum_l\binom{i_1}lX_1^lX_2^{i_1-l}X_2^{i_2}\cdots X_n^{i_n}\cr
&=\sum_{i_1+i_2-(q-1)\le l\le i_1}\binom{i_1}lX_1^lX_2^{i_1+i_2-l}X_3^{i_3}\cdots X_n^{i_n}.
\end{align*}
Fix $l$ such that $i_1+i_2-(q-1)\le l\le i_1$ and $\binom{i_1}l\ne 0$ and let 
\[
\jj=(l,i_1+i_2-l,i_3,\dots,i_n)\in\Omega_{q,n,r}.
\]
We want to show that $T(\jj)\le T(\ii)$. Let $T(\ii)=(t_0,\dots,t_{m-1})$ and $T(\jj)=(t_0',\dots,t_{m-1}')$. Let $0\le k\le m-1$. Write $i_1=ap^{k+1}+b$ and $l=up^{k+1}+v$, where $a,b,u,v\in\Bbb Z$ and $0\le b,v<p^{k+1}$. Since $\binom{i_1}l\ne 0$, by Lucas's theorem, $v\le b$. Thus
\begin{align*}
&\left\lfloor\frac{l}{p^{k+1}}\right\rfloor+\left\lfloor\frac{i_1+i_2-l}{p^{k+1}}\right\rfloor+\left\lfloor\frac{i_3}{p^{k+1}}\right\rfloor+\cdots+\left\lfloor\frac{i_n}{p^{k+1}}\right\rfloor\cr
&=\left\lfloor\frac{up^{k+1}+v}{p^{k+1}}\right\rfloor+\left\lfloor\frac{(a-u)p^{k+1}+b-v+i_2}{p^{k+1}}\right\rfloor+\left\lfloor\frac{i_3}{p^{k+1}}\right\rfloor+\cdots+\left\lfloor\frac{i_n}{p^{k+1}}\right\rfloor\cr
&\ge u+a-u+\left\lfloor\frac{i_2}{p^{k+1}}\right\rfloor+\left\lfloor\frac{i_3}{p^{k+1}}\right\rfloor+\cdots+\left\lfloor\frac{i_n}{p^{k+1}}\right\rfloor\cr
&=\left\lfloor\frac{i_1}{p^{k+1}}\right\rfloor+\cdots+\left\lfloor\frac{i_n}{p^{k+1}}\right\rfloor.
\end{align*}
This means, in light of \eqref{tk+=r}, that $t_k'\le t_k$. Therefore, $T(\jj)\le T(\ii)$.
\end{proof}

For $\ii\in\Omega_{q,n,r}$, we describe an operation on the matrix $D(\ii)$ called {\em digit transfer}\,: Take two entries $i_{j_1,k}$ and $i_{j_2,k}$ in the same column with $i_{j_1,k}>0$ and $i_{j_2,k}<p-1$. Replace $i_{j_1,k}$ with $i_{j_1,k}-1$ and $i_{j_2,k}$ with $i_{j_2,k}+1$.

\begin{lem}\label{digit-trans}
Let $M$ be a GL-submodule of $H_q(r,n)$ such that $X^\ii\in M$, where $\ii\in\Omega_{q,n,r}$. If $\ii'\in\Omega_{q,n,r}$ is such that $D(\ii')$ can be obtained from $D(\ii)$ through a digit transfer, then $X^{\ii'}\in M$.
\end{lem}

\begin{proof} 
This follows from Lemma~\ref{L4}.
\end{proof} 

\begin{lem}\label{L3.6}
Let $M$ be a GL-submodule of $H_q(r,n)$ such that $X^\ii\in M$, where $\ii\in\Omega_{q,n,r}$. Then $X^{\ii'}\in M$ for all $\ii'\in\Omega_{q,n,r}$ with $T(\ii')\le T(\ii)$.
\end{lem}

\begin{proof}
Let $\ii=(i_1,\dots,i_n)$ and $\ii'=(i_1',\dots,i_n')$ and let
\[
i_j=\sum_{k=0}^{m-1} i_{jk}p^k\quad \text{and}\quad i_j'=\sum_{k=0}^{m-1} i_{jk}'p^k
\]
be the base $p$ expansions of $i_j$ and $i_j'$, respectively, that is, $D(\ii)=(i_{jk})$ and $D(\ii')=(i_{jk}')$. 

\medskip

$1^\circ$ First assume that $T(\ii')= T(\ii)$. Since $D(\ii')$ and $D(\ii)$ have the same column sums, $D(\ii')$ can be obtained from $D(\ii)$ through a finite number of digit transfers. Therefore, by Lemma~\ref{digit-trans}, $X^{\ii'}\in M$.

\medskip

$2^\circ$
Now assume that $T(\ii')\lneqq T(\ii)$. Using induction on the partial order $\le$, it suffices to show that there exists $\ii''\in \Omega_{q,n,r}$ such that $X^{\ii''}\in M$ and $T(\ii')\le T(\ii'')\lneqq T(\ii)$.
Write $T(\ii)=(t_0,\dots,t_{m-1})$ and $T(\ii')=(t_0',\dots,t_{m-1}')$ and assume that $t_k=t_k'$ for $0\le k< l$ but $t_l>t_l'$. 

Let 
\[
(s_0,\dots,s_{m-1})=\left[\begin{matrix}1&\cdots&1\end{matrix}\right]
D(\ii)
\]
and 
\[
(s_0',\dots,s_{m-1}')=\left[\begin{matrix}1&\cdots&1\end{matrix}\right]
D(\ii').
\]
We claim that $s_l\ge p$. Since $s_lp^l=t_l-t_{l-1}$ and $s_l'p^l=t_l'-t_{l-1}'$, we have $s_l-s_l'=p^{-l}(t_l-t_l')$. It follows that $s_l-s_l'>0$ and $s_l-s_l'\equiv 0\pmod p$ since $t_l\equiv r\equiv t_l'\pmod{p^{l+1}}$ (by \eqref{TOmega}). Thus $s_l\ge p$.

We claim that there is some $k$ with $l<k\le m-1$, such that
\begin{equation}\label{ne}
\left[\begin{matrix}i_{1k}\cr\vdots\cr i_{nk}\end{matrix}\right]\ne\left[\begin{matrix}p-1\cr\vdots\cr p-1\end{matrix}\right].
\end{equation}
Otherwise, $s_k\ge s_k'$ for all $l<k\le m-1$. Then
\begin{equation}\label{tm-1}
t_{m-1}=t_l+s_{l+1}p^{l+1}+\cdots+s_{m-1}p^{m-1}>t_l'+s_{l+1}'p^{l+1}+\cdots+s_{m-1}'p^{m-1}=t_{m-1}',
\end{equation}
which is impossible since $t_{m-1}=t_{m-1}'=r$. Let $u$ be the smallest $k$ ($l<k\le m-1$) satisfying \eqref{ne}. Then, through digit transfers, we may write
\[
\left[\begin{matrix}
i_{1l}&\cdots&i_{1u}\cr
\vdots&&\vdots\cr
i_{nl}&\cdots&i_{nu}\end{matrix}\right]=\left[\begin{matrix}
p-1,&p-1&\cdots&p-1&i_{1u}\cr
i_{2l}&p-1&\cdots&p-1&i_{2u}\cr
\vdots&\vdots&&\vdots&\vdots\cr
i_{nl}&p-1&\cdots&p-1&i_{nu}
\end{matrix}\right],
\]
where $i_{2l}>0$ and $i_{1u}<p-1$. Similar to \eqref{tm-1}, we have $t_k>t_k'$ for $l\le k<u$. Let 
\[
\ii''=(i_1+p^l,i_2-p^l,i_3,\dots,i_n)\in\Omega_{q,n,r}.
\]
By Lemma~\ref{L4}, $X^{\ii''}\in M$. Write $D(\ii'')=(i_{jk}'')$. Then $i_{jk}''=i_{jk}$ for $0\le k<l$ and $u<k\le m-1$. For $l\le k\le u$, we have
\[
\left[\begin{matrix}
i_{1l}''&\cdots&i_{1u}''\cr
\vdots&&\vdots\cr
i_{nl}''&\cdots&i_{nu}''\end{matrix}\right]
=\left[\begin{matrix}
0&0&\cdots&0&i_{iu}+1\cr
i_{2l}-1&p-1&\cdots&p-1&i_{2u}\cr
\vdots&\vdots&&\vdots&\vdots\cr
i_{nl}&p-1&\cdots&p-1&i_{nu}
\end{matrix}\right].
\]
Therefore, 
\begin{align*}
&\left[\begin{matrix}1&\cdots&1\end{matrix}\right]D(\ii'')\cr
&=(s_0,\dots,s_{l-1},s_l-p,s_{l+1}-(p-1),\dots,s_{u-1}-(p-1),s_u+1,s_{u+1},\dots,s_{m-1}).
\end{align*}
Hence $T(\ii'')=(t_0'',\dots,t_m'')$, where
\[
t_k''=\begin{cases}
t_k-(p^{l+1}+(p-1)p^{l+1}+\cdots+(p-1)p^k)=t_k-p^{k+1}\quad\text{if}\ l\le k<u,\cr
t_k\quad\text{if}\ 0\le k<l\ \text{or}\ u\le k\le m-1.
\end{cases}
\]
Clearly, $T(\ii'')\lneqq T(\ii)$.
Recall that for $l\le k<u$, $t_k>t_k'$ and $t_k\equiv  t_k'\pmod{p^{k+1}}$, whence $t_k\ge t_k'+p^{k+1}$. Therefore $t_k''\ge t_k'$ for all $k$, i.e., $T(\ii'')\ge T(\ii')$. This completes the proof of the lemma. 
\end{proof}

\begin{defn}
A subset $I\subset T(\Omega_{q,n,r})$ is called an ideal of the partially ordered set $(T(\Omega_{q,n,r}),\le)$ if for $\bt,\bt'\in T(\Omega_{q,n,r})$ with $\bt'\le \bt$, $\bt\in I$ implies $\bt'\in I$.
\end{defn}

For each ideal $I$ of $(T(\Omega_{q,n,r}),\le)$, define
\begin{equation}\label{M(I)}
M(I)=\text{the $\f_q$-linear span of}\ \{X^\ii:\ii\in T^{-1}(I)\}.
\end{equation}
Let $\mathcal I$ be the set of all ideals of $(T(\Omega_{q,n,r}),\le)$ and $\mathcal M$ be the set of all $\text{GL}(n,\f_q)$-submodules of $H_q(r,n)$.
Combining several previous lemmas, we arrive at the following main result.

\begin{thm}\label{T3.8}
The map 
\[
\begin{array}{cccc}
\Phi:&\mathcal I&\longrightarrow&\mathcal M\vspace{0.3em}\cr
&I&\longmapsto& M(I)
\end{array}
\]
is a bijection.
\end{thm}

\begin{proof}
For each $M\in\mathcal M$, define
\begin{equation}\label{Psi}
\Psi(M)=\{T(\ii):\ii\in\Omega_{q,n,r}, X^\ii\in M\}.
\end{equation}
By Lemma~\ref{L3.6}, $\Psi(M)$ is an ideal of $(T(\Omega_{q,n,r}),\le)$. Hence we have a map $\Psi:\mathcal M\to \mathcal I$. It remains to show that both $\Phi\circ\Psi$ and $\Psi\circ\Phi$ are identity maps.

Let $M\in\mathcal M$. If $\ii\in T^{-1}(\Psi(M))$, then $T(\ii)\in\Psi(M)$. By \eqref{Psi}, $T(\ii)=T(\jj)$ for some $\jj\in\Omega_{q,n,r}$ with $X^\jj\in M$. By Lemma~\ref{L3.6}, $X^\ii\in M$. Since $M(\Psi(M))$ is the $\f_q$-linear span of $\{X^\ii:\ii\in T^{-1}(\Psi(M))\}$, we have $M(\Psi(M))\subset M$. On the other hand, if $X^\ii$ is a monomial in $M$, by \eqref{Psi}, $T(\ii)\in\Psi(M)$, i.e., $\ii\in T^{-1}(\Psi(M))$. Then by \eqref{M(I)}, $X^\ii\in M(\Psi(M))$. Since $M$ is generated over $\f_q$ by a set of monomials (Corollary~\ref{C3.4}), we have $M\subset M(\Psi(M))$. Therefore, $M(\Psi(M))=M$ for all $M\in\mathcal M$, whence $\Phi\circ\Psi$ is the identity map.

Let $I\in\mathcal I$. If $\bt\in\Psi(M(I))$, then $\bt=T(\ii)$ for some $\ii\in\Omega_{q,n,r}$ with $X^\ii\in M(I)$. By \eqref{M(I)}, $\ii\in T^{-1}(I)$. Thus $\bt=T(\ii)\in I$. So $\Psi((M(I))\subset I$. On the other hand, if $\bt\in I$, choose $\ii\in T^{-1}(\bt)$. Then by \eqref{M(I)}, $X^\ii\in M(I)$. By \eqref{Psi}, $\bt=T(\ii)\in\Psi(M(I))$. So $I\subset\Psi(M(I))$. Therefore, $\Psi(M(I))=I$ for all $I\in\mathcal I$, whence $\Psi\circ\Phi$ is the identity map.
\end{proof}

Each ideal $I$ of the partially ordered set $(T(\Omega_{q,n,r}),\le)$ is determined by its {\em boundary} which is the set of maximal elements in $I$. On the other hand, each subset $B$ of pairwise noncomparable elements of $T(\Omega_{q,n,r})$ determines an ideal with boundary $B$. Therefore, ideals of $(T(\Omega_{q,n,r}),\le)$ are in one-to-one correspondence with subsets of pairwise noncomparable elements of $T(\Omega_{q,n,r})$. Consequently, the enumeration of $\text{\rm GL}(n,\f_q)$-submodules of $H_q(r,n)$ is equivalent to the enumeration of subsets of pairwise noncomparable elements of $T(\Omega_{q,n,r})$. Let $\mathcal B$ denote the set of all of subsets of pairwise noncomparable elements of $T(\Omega_{q,n,r})$. For $B\in \mathcal B$, the corresponding ideal of $(T(\Omega_{q,n,r}),\le)$ is $I=\{\bt\in T(\Omega_{q,n,r}):\bt\le \bt'\ \text{for some}\ \bt'\in B\}$ and the corresponding GL-module $M(I)$ is the $\f_q$-span of all $X^{\ii}$ such that $T(\ii)\le \bt'$ for some $\bt'\in B$.

\begin{exmp}\rm
Let $p=2$, $m=3$, $q=8$, $n=4$ and $r=8$. Then
\[
T(\Omega_{8,4,8})=\{(t_0,t_1,8): (t_0,t_1)\ \text{given in Figure~\ref{F2}}\}.
\]
\begin{figure}
\[
\beginpicture
\setcoordinatesystem units <1.5em,1.5em> point at 0 0

\arrow <5pt> [.2,.67] from 0 0 to 0 10 
\arrow <5pt> [.2,.67] from 0 0 to 8 0 

\setdashes<1mm>
\plot 2 0  2 8 /
\plot 4 0  4 8 /
\plot 0 4  4 4 /
\plot 0 8  4 8 /

\put {$\bullet$} at 0 0
\put {$\bullet$} at 0 4
\put {$\bullet$} at 0 8
\put {$\bullet$} at 2 4
\put {$\bullet$} at 2 8
\put {$\bullet$} at 4 4
\put {$\bullet$} at 4 8

\put {$\scriptstyle t_0$} [t] at 8 -0.2
\put {$\scriptstyle t_1$} [r] at -0.2 10
\put {$\scriptstyle 4$} [r] at -0.2 4
\put {$\scriptstyle 8$} [r] at -0.2 8
\put {$\scriptstyle 2$} [t] at 2 -0.2
\put {$\scriptstyle 4$} [t] at 4 -0.2
\put {$\scriptstyle 0$} [t] at 0 -0.2

\endpicture
\]
\caption{$T(\Omega_{8,4,8})$}\label{F2}
\end{figure}

\noindent
The boundaries of the ideals of $(T(\Omega_{8,4,8}),\le)$ are given in Table~\ref{Tb1}, where $(t_0,t_1)$ stands for $(t_0,t_1,8)$.
\begin{table}
\caption{Ideals of $(T(\Omega_{8,4,8}),\le)$}\label{Tb1}
   \renewcommand*{\arraystretch}{1.2}
    \centering
     \begin{tabular}{l|l}       
         \hfil boundary  & \hfil ideal \\ \hline
         $\emptyset$ & $I_0=\emptyset$\\
         $\{(0,0)\}$ & $I_1=\{(0,0)\}$\\
         $\{(0,4)\}$ & $I_2=\{(0,0),(0,4)\}$\\
         $\{(2,4)\}$ & $I_3=\{(0,0),(0,4),(2,4)\}$\\
         $\{(0,8)\}$ & $I_4=\{(0,0),(0,4),(0,8)\}$\\
         $\{(4,4)\}$ & $I_5=\{(0,0),(0,4),(2,4),(4,4)\}$\\
         $\{(0,8),(2,4)\}$ & $I_6=\{(0,0),(0,4),(0,8),(2,4)\}$\\                 
         $\{(2,8)\}$ & $I_7=\{(0,0),(0,4),(0,8),(2,4),(2,8)\}$\\                  
         $\{(0,8),(4,4)\}$ & $I_8=\{(0,0),(0,4),(0,8),(2,4),(4,4)\}$\\
         $\{(2,8),(4,4)\}$ & $I_9=\{(0,0),(0,4),(0,8),(2,4),(2,8),(4,4)\}$\\
         $\{(4,8)\}$ & $I_{10}=\{(0,0),(0,4),(0,8),(2,4),(2,8),(4,4),(4,8)\}$\\
        
     \end{tabular}
\end{table}
Elements of $T^{-1}(I_j)$, $0\le j\le 10$, are given in Table~\ref{Tb2}, where $i_1i_2i_3i_4$ stands for $(i_1,i_2,i_3,i_4)$ and their permutations.
\begin{table}
\caption{$T^{-1}(I_j)$, $0\le j\le 10$}\label{Tb2}
   \renewcommand*{\arraystretch}{1.2}
    \centering
     \begin{tabular}{l|l}       
         \hfil $j$  & \hfil  elements of $T^{-1}(I_j)$ \\ \hline       
         0& $\emptyset$ \\
         1& 4400 \\
         2& 4400, 6200, 4220 \\
         3& 4400, 6200, 4220, 7100, 6110, 5300, 5210, 4310, 4211 \\
         4& 4400, 6200, 4220, 2222 \\
         5& 4400, 6200, 4220, 7100, 6110, 5300, 5210, 4310, 4211, 5111  \\
         6& 4400, 6200, 4220, 2222, 7100, 6110, 5300, 5210, 4310, 4211 \\
         7& 4400, 6200, 4220, 2222, 7100, 6110, 5300, 5210, 4310, 4211, 3320, 3221 \\
         8& 4400, 6200, 4220, 2222, 7100, 6110, 5300, 5210, 4310, 4211, 5111  \\ 
         9& 4400, 6200, 4220, 2222, 7100, 6110, 5300, 5210, 4310, 4211, 3320, 3221, 5111 \\    
         10& $\Omega_{8,4,8}$ \\                      
     \end{tabular}
\end{table}
The bases of the submodules $M(I_j)$ follow from Table~\ref{Tb2} immediately. For example, $M(I_4)$ is the $\f_8$-span of $X_1^4X_2^4,X_1^6X_2^2,X_1^4X_2^2X_3^2,X_1^2X_2^2X_3^2X_4^2$ and their permutations.

\end{exmp}

\section{Factors of the Composition Series of $H_q(r,n)$}

We follow the notation introduced in Section~3. In addition, for $I,I'\in\mathcal I$,  we write $I\subset_{\max} I'$ to mean  that $I\subsetneq I'$ and there is no $I''\in \mathcal I$ such that $I\subsetneq I''\subsetneq I'$.

A composition series of $H_q(r,n)$ is given by 
\begin{equation}\label{comp}
M(I_0)\subset M(I_1)\subset\cdots\subset M(I_N),
\end{equation}
where $I_0,I_1,\dots,I_N\in\mathcal I$ are such that
\begin{equation}\label{ideals}
\emptyset=I_0\subset_{\max} I_1\subset_{\max}\cdots \subset_{\max} I_N=T(\Omega_{q,n,r}).
\end{equation}
It is clear that $I_i\subset_{\max}I_{i+1}$ if and only if $I_i=I_{i+1}\setminus\{\bt\}$ for some maximal element $\bt$ in $I_{i+1}$. Therefore, the composition series \eqref{comp} can be obtained as follows: First, let $I_N=T(\Omega_{q,n,r})$. Choose a maximal element $\bt_N\in I_N$ and let $I_{N-1}=I_N\setminus\{\bt_N\}$. Choose a maximal element $\bt_{N-1}$ in $I_{N-1}$ and let $I_{N-2}=I_{N-1}\setminus\{\bt_{N-1}\}$. Continue this way until $I_0=\emptyset$. Clearly, $\bt_N,\bt_{N-1},\dots,\bt_1$ is an enumeration of all elements in $T(\Omega_{n,q,r})$, whence $N=|T(\Omega_{q,n,r})|$. The factors of the composition series \eqref{comp} are
\[
M(I_i)\,/\,M(I_{i-1})=M(I_i)\,/\,M(I_i\setminus\{\bt_i\}),\quad 1\le i\le N.
\]
The structure of the module $M(I_i)\,/\,M(I_i\setminus\{\bt_i\})$ depends only on $\bt_i$ but not on $I_i$. For $\bt\in T(\Omega_{q,n,r})$, let
\[
I(\bt)=\{\bt'\in T(\Omega_{q,n,r}):\bt'\le \bt\}\in\mathcal I.
\]

\begin{lem}\label{L4.1}
Let $I\in\mathcal I$ and $\bt$ be a maximal element of $I$. Then
\[
M(I)\,/\,M(I\setminus\{\bt\})\cong M(I(\bt))\,/\,M(I(\bt)\setminus\{\bt\}).
\]
\end{lem}

\begin{proof}
Define a GL-module map
\[
\begin{array}{cccc}
\phi:&M(I(\bt))&\longrightarrow &M(I)\,/\,M(I\setminus\{\bt\})\vspace{0.2em}\cr
&f&\longmapsto &f+M(I\setminus\{\bt\}).\cr
\end{array}
\]
Since $M(I)=M(I\setminus\{\bt\})+M(I(\bt))$, $\phi$ is onto. Since $I(\bt)\setminus\{\bt\}\subset I\setminus\{\bt\}$, we have $M(I(\bt)\setminus\{\bt\})\subset M(I\setminus\{\bt\})$, whence $M(I(\bt)\setminus\{\bt\})\subset\ker\phi$. Thus $\phi$ induces an onto GL-module map
\[
\bar\phi:M(I(\bt))\,/\,M(I(\bt)\setminus\{\bt\})\longrightarrow M(I)\,/\,M(I\setminus\{\bt\}).
\]
Since $I(\bt)\setminus\{\bt\}\subset_{\max}I(\bt)$, $M(I(\bt))\,/\,M(I(\bt)\setminus\{\bt\})$ is an irreducible GL-module. It follows that $\bar\phi$ is an isomorphism. 
\end{proof}

For $\bt\in T(\Omega_{q,n,r})$, let
\[
\frak M(\bt)=M(I(\bt))\,/\,M(I(\bt)\setminus\{\bt\}).
\]
The structure of the GL-module $\frak M(\bt)$ is easy to describe (with a little abuse of notation). It has a basis $\{X^\ii: \ii\in T^{-1}(\bt)\}$ over $\f_q$. When $A\in\text{GL}(n,\f_q)$ acts on $X^\ii$, in the expansion \eqref{AXi} of $A(X^\ii)$, only the terms $\alpha_\jj X^\jj$ with $\jj\in T^{-1}(\bt)$ are kept. We have $\dim_{\f_q}\frak M(\bt)=|T^{-1}(\bt)|$, which can be made explicit. 

\begin{lem}\label{L-dim}
Let $\bt=(t_0,\dots,t_{m-1})\in T(\Omega_{q,n,r})$ and 
\begin{equation}\label{s-t}
(s_0,\dots,s_{m-1})=\Bigl(t_0,\frac{t_1-t_0}p,\dots,\frac{t_{m-1}-t_{m-2}}{p^{m-1}}\Bigr).
\end{equation}
Then
\begin{equation}\label{dim}
\dim_{\f_q}\frak M(\bt)=|T^{-1}(\bt)|=\prod_{j=0}^{m-1}\Bigl(\sum_{0\le k\le s_j/p}(-1)^k\binom nk\binom{n-1+s_j-kp}{n-1}\Bigr).
\end{equation}
\end{lem}

\begin{proof}
It is straightforward from \eqref{s-t} that 
\begin{equation}\label{t-s}
(t_0,\dots,t_{m-1})=(s_0,\dots,s_{m-1})\left[\begin{matrix}p^0&p^0&\cdots&p^0\cr &p^1&\cdots&p^1\cr &&\ddots&\vdots\cr &&&p^{m-1}\end{matrix}\right].
\end{equation}
By \eqref{Ti}, $\ii\in T^{-1}(\bt)$ if and only if 
\[
\left[\begin{matrix}1&\cdots&1\end{matrix}\right]
D(\ii)
\left[\begin{matrix}p^0&p^0&\cdots&p^0\cr &p^1&\cdots&p^1\cr &&\ddots&\vdots\cr &&&p^{m-1}\end{matrix}\right]=T(\ii)=(t_0,\dots,t_{m-1}).
\]
In light of \eqref{t-s}, this happens if and only if
\begin{equation}\label{4.1}
\left[\begin{matrix}1&\cdots&1\end{matrix}\right]D(\ii)=(s_0,\dots,s_{m-1}).
\end{equation}
Write
\[
D(\ii)=
\left[\begin{matrix}i_{10}&\cdots&i_{1,m-1}\cr\vdots&&\vdots\cr i_{n0}&\cdots&i_{n,m-1}\end{matrix}\right].
\]
Then \eqref{4.1} is satisfied if and only if for all $0\le j\le m-1$,
\begin{equation}\label{4.2}
i_{1j}+\cdots+i_{nj}=s_j,\quad 0\le i_{1j},\dots,i_{nj}\le p-1.
\end{equation}
The number of $(i_{1j},\dots,i_{nj})$ satisfying \eqref{4.2}, denoted by $N_j$, is the coefficient of $X^{s_j}$ in $(1+X+\cdots+X^{p-1})^n$. We have
\begin{align*}
&(1+X+\cdots+X^{p-1})^n=\Bigl(\frac{1-X^p}{1-X}\Bigr)^n=(1-X^p)^n(1-X)^{-n}\cr
&=\Bigl(\sum_k\binom nk(-X^p)^k \Bigr)\Bigl(\sum_l\binom{-n}l(-X)^l\Bigr)=\sum_{k,l}\binom nk\binom{-n}l(-1)^{k+l}X^{kp+l}.
\end{align*}
Hence
\begin{align*}
N_j\,&=\sum_{kp+l=s_j}\binom nk\binom{-n}l(-1)^{k+l}\cr
&=\sum_{kp+l=s_j}(-1)^k\binom nk\binom{n+l-1}l\kern2em (\text{since}\ (-1)^l\binom{-n}l=\binom{n+l-1}l\,)\cr
&=\sum_{\substack{0\le k\le n\cr s_j-kp\ge 0}}(-1)^k\binom nk\binom{n+s_j-kp-1}{n-1}.
\end{align*}
By \eqref{TOmega} and \eqref{s-t}, $s_j/p\le n(p-1)/p<n$, whence $k\le s_j/p$ implies $k\le n$. Therefore, the effective range for $k$ in the above sum is $0\le k\le s_j/p$. Now,
\[
|T^{-1}(\bt)|=N_0\cdots N_{m-1}=\prod_{j=0}^{m-1}\Bigl(\sum_k(-1)^k\binom nk\binom{n-1+s_j-kp}{n-1}\Bigr).
\]
\end{proof}

\begin{lem}\label{L4.3}
If $\bt,\bt'\in T(\Omega_{q,n,r})$ are such that $\frak M(\bt)\cong\frak M(\bt')$, then $\bt=\bt'$.
\end{lem}

\begin{proof}
If $r=n(q-1)$, then $\bt=\bt'=T(\ii)$, where
\[
D(\ii)=\left[\begin{matrix}p-1&\cdots &p-1\cr \vdots&&\vdots\cr p-1&\cdots&p-1\end{matrix}\right].
\]
So we assume that $r<n(q-1)$.

We use induction on $n$. When $n=1$, let $\ii\in T^{-1}(\bt)$, $\ii'\in T^{-1}(\bt')$, and write $D(\ii)=(i_{10},\dots,i_{1,m-1})$, $D(\ii')=(i_{10}',\dots,i_{1,m-1}')$, where $0\le i_{1j},i_{1j}'\le p-1$, $0\le j\le m-1$. Then
\[
i_{10}p^0+\cdots+i_{1,m-1}p^{m-1}=r=i_{10}'p^0+\cdots+i_{1,m-1}'p^{m-1},
\]
whence $D(\ii)=D(\ii')$. Therefore, $\ii=\ii'$, and hence $\bt=\bt'$.

Now assume $n>0$. Since $\frak M(\bt)$ is generated over $\f_q$ by $X^\ii$, $\ii\in T^{-1}(\bt)$, we have
\[
\frak M(\bt)=X_n^0M_0+\cdots+X_n^{q-1}M_{q-1},
\]
where $M_k$ is generated over $\f_q$ by $(X_1,\dots,X_{n-1})^\jj$ with $\jj\in\Omega_{q,n-1,r-k}$ such that $(\jj,k)\in T^{-1}(\bt)$. In the same way, 
\[
\frak M(\bt')=X_n^0M_0'+\cdots+X_n^{q-1}M_{q-1}',
\]
where $M_k'$ is generated over $\f_q$ by $(X_1,\dots,X_{n-1})^\jj$ with $\jj\in\Omega_{q,n-1,r-k}$ such that $(\jj,k)\in T^{-1}(\bt')$.  Let $f:\frak M(\bt)\to\frak M(\bt')$ be the given isomorphism. We claim that
\begin{equation}\label{claim}
f(X_n^kM_k)\subset X_n^kM_k',\quad 0\le k\le q-2.
\end{equation}
Let $\alpha\in X_n^kM_k$, where $0\le k\le q-2$, and write
\[
f(\alpha)=\beta_0+\cdots+\beta_{q-1},
\]
where $\beta_k\in X_n^kM_k'$. Let $A\in\text{GL}(n,\f_q)$ be such that 
\[
(X_1,\dots,X_n)A=(X_1,\dots,X_{n-1},\lambda X_n).
\]
Then
\begin{align*}
\lambda^kf(\alpha)\,&=f(\lambda^k\alpha)=f(A(\alpha))=A(f(\alpha))\cr
&=A(\beta_0+\cdots+\beta_{q-1})=\lambda^0\beta_0+\cdots+\lambda^{q-1}\beta_{q-1}.
\end{align*}
Since this is true for all $\lambda\in\f_q^*$, we have
\[
f(\alpha)=\begin{cases}
\beta_k&\text{if}\ 1\le k\le q-2,\cr
\beta_0+\beta_{q-1}&\text{if}\ k=0.
\end{cases}
\]
We only have to show that when $k=0$, $\beta_{q-1}=0$. Assume to the contrary that $\beta_{q-1}\ne 0$. Write $\beta_{q-1}=X_n^{q-1}u$, where $0\ne u\in M_{q-1}'$. Since $r<n(q-1)$, $\deg u<(n-1)(q-1)$. Let $1\le i\le n-1$ and let $A\in\text{GL}(n,\f_q)$ be such that 
\[
(X_1,\dots,X_n)A=(X_1,\dots,X_{n-1},X_n-X_i).
\]
Then
\begin{align*}
0\,&=f(A(\alpha))-f(\alpha)=A(\beta_0+X_n^{q-1}u)-(\beta_0+X_n^{q-1}u)\cr
&=\bigl((X_n-X_i)^{q-1}-X_n^{q-1}\bigr)u=(X_i^{q-1}+X_i^{q-2}X_n+\cdots+X_iX_n^{q-2})u.
\end{align*}
It follows that $u=X_i^{q-1}u_i$ for some homogeneous polynomial $u_i$ in $X_1,\dots,X_{n-1}$. Since this is true for all $1\le i\le n-1$, we have $u=X_1^{q-1}\cdots X_{n-1}^{q-1}u'$ for some homogeneous polynomial $u_i$ in $X_1,\dots,X_{n-1}$. This is impossible since $\deg u<(n-1)(q-1)$.
Hence \eqref{claim} is proved.

By symmetry, $f^{-1}(X_n^kM_k')\subset X_n^kM_k$ for $0\le k\le q-2$. Hence for $0\le k\le q-2$, the restriction $f:X_n^kM_k\to X_n^kM_k'$ is an $\f_q$-isomorphism. For $\alpha\in M_k$, write $f(X_n^k\alpha)=X_n^kf_k(\alpha)$, where $f_k(\alpha)\in M_k'$. Then $f_k: M_k\to M_k'$ is a $\text{GL}(n-1,\f_q)$-module isomorphism.

Let $\ii\in T^{-1}(\bt)$ and write
\[
D(\ii)=
\left[\begin{matrix}i_{10}&\cdots&i_{1,m-1}\cr\vdots&&\vdots\cr i_{n0}&\cdots&i_{n,m-1}\end{matrix}\right].
\]
Since $r<n(q-1)$, we may assume that $(i_{n0},\dots,i_{n,m-1})\ne(p-1,\dots,p-1)$. Let $k=i_{n0}p^0+\cdots+i_{n,m-1}p^{m-1}$. Then $0\le k\le q-2$. We have $M_k=\frak M(\btau)$ and $M_k'=\frak M(\btau')$, where
\[
\btau=\bt-(i_{n0},\dots,i_{n,m-1})\left[\begin{matrix}p^0&p^0&\cdots&p^0\cr &p^1&\cdots&p^1\cr &&\ddots&\vdots\cr &&&p^{m-1}\end{matrix}\right]\in T(\Omega_{q,n-1,r-k})
\]
and  
\[
\btau'=\bt'-(i_{n0},\dots,i_{n,m-1})\left[\begin{matrix}p^0&p^0&\cdots&p^0\cr &p^1&\cdots&p^1\cr &&\ddots&\vdots\cr &&&p^{m-1}\end{matrix}\right]\in T(\Omega_{q,n-1,r-k});
\]
these claims follow from the definitions of $M_k$, $M_k'$, $\frak M(\btau)$ and $\frak M(\btau')$. Since $\frak M(\btau)\cong\frak M(\btau')$, by the induction hypothesis, $\btau=\btau'$, whence $\bt=\bt'$.
\end{proof}

We summarize the facts about the composition series of $H_q(r,n)$ in the following theorem.

\begin{thm}\label{T4.4}
The composition factors of $H_q(r,n)$ are $\frak M(\bt)$, $\bt\in T(\Omega_{q,n,r})$, each appearing exactly once. These factors are pairwise nonisomorphic and their dimensions are given in \eqref{dim}. The length of the composition series of $H_q(r,n)$ is $|T(\Omega_{q,n,r})|$.
\end{thm}

There does not seem to be an explicit formula for the number $|T(\Omega_{q,n,r})|$. However, the generating function $\sum_r|T(\Omega_{q,n,r})|X^r$ can be easily determined. By \eqref{Ti}, we have
\begin{align*}
&|T(\Omega_{q,n,r})|\cr
=\,&|\{(s_0,\dots,s_{m-1})\in\Bbb N^m: 0\le s_i\le n(p-1),\ s_0p^0+\cdots+s_{m-1}p^{m-1}=r\}|\cr
=\,&\text{the coefficient of $X^r$ in}\ \prod_{k=0}^{m-1}(1+X^{p^k}+X^{2p^k}+\cdots+X^{n(p-1)p^k}).
\end{align*}
Hence
\[
\sum_r|T(\Omega_{q,n,r})|X^r=\prod_{k=0}^{m-1}\frac{1-X^{(n(p-1)+1)p^k}}{1-X^{p^k}}.
\]
The length of a composition series of $\mathcal F(\f_q^n,\f_q)$ (as a GL-module or as an AGL-module) is
\[
\sum_r|T(\Omega_{q,n,r})|=\bigl(n(p-1)+1)^m.
\]
In comparison, the ascending chain of Reed-Muller codes 
\[
\{0\}=R_q(-1,n)\subset R_q(0,n)\subset\cdots\subset R_q(n(q-1),n)=\mathcal F(\f_q^n,\f_q)
\]
has length $n(q-1)+1$.

Finally, we address the following question: If $\bt_1\in T(\Omega_{q,n,r_1})$ and  $\bt_2\in T(\Omega_{q,n,r_2})$, where $r_1\ne r_2$, is it possible that $\frak M(\bt_1)\cong\frak M(\bt_2)$? If $r_1=0$ and $r_2=n(q-1)$, then $\frak M(\bt_1)$ is the 1-dimensional $\f_q$-space generated by $1$, $\frak M(\bt_2)$ is the 1-dimensional $\f_q$-space generated by $X_1^{q-1}\cdots X_n^{q-1}$, and $\text{GL}(n,\f_q)$ acts trivially on both $\frak M(\bt_1)$ and $\frak M(\bt_2)$. Therefore $\frak M(\bt_1)\cong\frak M(\bt_2)$. However, this is the only case where an isomorphism occurs.

\begin{thm}\label{T4.5}
Let $0\le r_1<r_2\le n(q-1)$ be such that $(r_1,r_2)\ne(0,n(q-1))$ and let $\bt_1\in T(\Omega_{q,n,r_1})$ and $\bt_2\in T(\Omega_{q,n,r_2})$. Then $\frak M(\bt_1)\not\cong\frak M(\bt_2)$.
\end{thm}

\begin{proof}
If $r_1=0$, then $0<r_2<n(q-1)$. It is easy to see that $\dim_{\f_q}\frak M(\bt_1)=1<\dim_{\f_q}\frak M(\bt_2)$, whence $\frak M(\bt_1)\not\cong\frak M(\bt_2)$. So assume $r_1>0$.

Assume to the contrary that there is an isomorphism $f:\frak M(\bt_1)\to\frak M(\bt_2)$. Let $\ii=(i_1,\dots,i_n)\in T^{-1}(\bt_1)$, whence $X^\ii\in\frak M(\bt_1)$. Write
\[
f(X^\ii)=\sum_{\jj\in T^{-1}(\bt_2)}\alpha_\jj X^\jj,\quad \alpha_\jj\in\f_q.
\]
Let $\epsilon$ be a primitive element of $\f_q$. Let $A\in\text{GL}(n,\f_q)$ be such that 
\[
(X_1,\dots,X_n)A=(\epsilon^{a_1}X_1,\dots,\epsilon^{a_n}X_n),
\]
where $(a_1,\dots,a_n)\in(\Bbb Z/(q-1)\Bbb Z)^n$. We have
\begin{align*}
&\epsilon^{a_1i_1+\cdots+a_ni_n}\sum_{\jj\in T^{-1}(\bt_2)}\alpha_\jj X^\jj\cr
=\,&\epsilon^{a_1i_1+\cdots+a_ni_n}f(X^\ii)=f(A(X^\ii))=A(f(X^\ii))\cr
=\,&\sum_{\jj=(j_1,\dots,j_n)\in T^{-1}(\bt_2)}\alpha_\jj \epsilon^{a_1j_1+\cdots+a_nj_n}X^\jj.
\end{align*}
If $\jj\not\equiv \ii\pmod{q-1}$, there exists $(a_1,\dots,a_n)\in(\Bbb Z/(q-1)\Bbb Z)^n$ such that $a_1i_1+\cdots+a_ni_n\not\equiv a_1j_1+\cdots+a_nj_n\pmod{q-1}$; it follows from the above that $\alpha_\jj=0$. Therefore, we have
\[
f(X^\ii)=\sum_{\substack{\jj\in T^{-1}(\bt_2)\cr \jj\equiv \ii\,(\text{mod}\,q-1)}}\alpha_\jj X^\jj.
\]

By Lemma~\ref{digit-trans}, we may replace $\ii$ with $\ii'$, where $D(\ii')$ is obtained from $D(\ii)$ through a {\em digit transfer}. Since $0<i_1+\cdots+i_n=r_1<n(q-1)$, by a digit transfer, we may assume that $0<i_1<q-1$. We may further assume that 
\[
i_j
\begin{cases}
\in\{1,\dots,q-2\}&\text{if}\ 1\le j\le k,\cr
=q-1&\text{if}\ k+1\le j\le l,\cr
=0&\text{if}\ l+1\le j\le n,
\end{cases}
\]
where $1\le k\le l\le n$. Then
\[
f(X^\ii)=X_1^{i_1}\cdots X_k^{i_k}g(X_{k+1}^{q-1},\dots,X_n^{q-1}),
\]
where $g(Y_1,\dots,Y_{n-k})\in\f_q[Y_1,\dots,Y_{n-k}]$ is homogeneous of degree $(r_2-i_1-\cdots-i_k)/(q-1)$ and $\deg_{Y_j}g\le 1$ for all $1\le j\le n-k$. We claim that $l<n$ and $\deg_{Y_j}g=1$ for some $l-k<j\le n-k$. Otherwise, 
\[
r_2=i_1+\cdots+i_k+(q-1)\deg g\le i_1+\cdots+i_k+(q-1)(l-k)=r_1,
\]
which is a contradiction. Without loss of generality, assume $\deg_{Y_{n-k}}g=1$. Then
\[
g(X_{k+1}^{q-1},\dots,X_n^{q-1})=X_n^{q-1}g_1(X_{k+1}^{q-1},\dots,X_{n-1}^{q-1})+g_2(X_{k+1}^{q-1},\dots,X_{n-1}^{q-1}),
\]
where $g_1,g_2\in\f_q[Y_1,\dots,Y_{n-k-1}]$ are homogeneous, $g_1\ne 0$, $\deg g_1=\deg g-1$, and $\deg_{Y_j}g_1\le 1$ for all $1\le j\le n-k-1$. Let $A\in\text{GL}(n,\f_q)$ be such that 
\[
(X_1,\dots,X_n)A=(X_1,\dots,X_{n-1},X_n-X_1).
\]
Then
\begin{align}\label{*}\\
&X_1^{i_1}\cdots X_k^{i_k}\bigl(X_n^{q-1}g_1(X_{k+1}^{q-1},\dots,X_{n-1}^{q-1})+g_2(X_{k+1}^{q-1},\dots,X_{n-1}^{q-1})\bigr)\cr
=\,&f(X^\ii)=f(A(X^\ii))=A(f(X^\ii))\cr
=\,&X_1^{i_1}\cdots X_k^{i_k}\bigl((X_n-X_1)^{q-1}g_1(X_{k+1}^{q-1},\dots,X_{n-1}^{q-1})+g_2(X_{k+1}^{q-1},\dots,X_{n-1}^{q-1})\bigr)\cr
=\,&X_1^{i_1}\cdots X_k^{i_k}\Bigl(\Bigl(\sum_{a=0}^{q-1}X_1^aX_n^{q-1-a}\Bigr)g_1(X_{k+1}^{q-1},\dots,X_{n-1}^{q-1})+g_2(X_{k+1}^{q-1},\dots,X_{n-1}^{q-1})\Bigr).\nonumber
\end{align}
Since $i_1<q-1$, there exists $0<a\le q-1-i_1$ such that the sum $a+i_1$ has no carry in base $p$. Let $\jj=(i_1+a,i_2,\dots,i_{n-1},q-1-a)$. Since $(i_1,\dots,i_{n-1},q-1)\in T^{-1}(\bt_2)$ (by assumption), we have $\jj\in T^{-1}(\bt_2)$. Since $X^\jj$ appears on the RHS of \eqref{*} but not on the LHS, we have a contradiction.
\end{proof}

\section{Irreducible Representations of $\text{GL}(n,\f_q)$ over $\f_q$}

The number of irreducible representations of $\text{GL}(n,\f_q)$ over $\f_q$ equals the number of $p$-regular $\f_q$-conjugacy classes of $\text{GL}(n,\f_q)$ (\cite{Berman-DANSSSR-1956, Reiner-PAMS-1964}). The $p$-regular $\f_q$-conjugacy classes of $\text{GL}(n,\f_q)$ are precisely the conjugacy classes of the elements whose elementary divisors are irreducible over $\f_q$. Such conjugacy classes are
parametrized by monic polynomials of degree $n$ over $\f_q$ with nonzero constant term. Therefore, the number of irreducible representations of $\text{GL}(n,\f_q)$ over $\f_q$ equals $q^{n-1}(q-1)$.

When $n=2$, the irreducible representations of $\text{GL}(n,\f_q)$ over $\f_q$ were determined by Brauer and Nesbitt \cite{Brauer-Nesbitt-AM-1941}; also see Barthel and Livn\'e \cite{Barthel-Livne-DMJ-1994}.

For an arbitrary $n$, James and Kerber \cite[Exercise~8.4]{James-Kerber-1981} outlined a method for constructing all irreducible $\f_q\text{GL}(n,\f_q)$-modules using Weyl modules by emulating a construction of irreducible modules over a certain superalgebra by Cater and Lusztig \cite{Carter-Lusztig-MZ-1974}.
However, the outlined construction in \cite{James-Kerber-1981} is not a straightforward adaptation of that of \cite{Carter-Lusztig-MZ-1974}; additional technical steps are needed to prove the claims in the construction of \cite{James-Kerber-1981}. References do not seem to be immediately available and we plan to give a detailed account of the construction in a separate paper. 
The irreducible $\f_q\text{GL}(n,\f_q)$-modules constructed from Weyl modules are not entirely explicit. For example, their dimensions are not known.

The factors $\frak M(\bt)$ of the composition series of $H_q(r,n)$ that we constructed in Section~4 only account for a small portion of the irreducible $\f_q\text{GL}(n,\f_q)$-modules. However, they are explicit, and in particular, their dimensions are known. The corresponding representations of these modules belong to the class of polynomial representations of the general linear group in the sense that the entries of their representation matrices are homogeneous polynomials in the entries of the elements of the general linear group. When $F$ is an infinite field, the irreducible polynomial representations of $\text{GL}(n,F)$ have been determined \cite{Green-LNM-1980}. However, when $F$ is finite, the knowledge of such representations is incomplete. 

\section{Conclusion}

In this paper, we considered two separate questions about the AGL-module structure of the quotient $H_q(r,n)=R_q(r,n)/R_q(r-1,n)$ of two consecutive Reed-Muller codes. In the first question, we proved a duality between $H_q(r,n)$ and $H_q(r',n)$, where $r+r'=n(q-1)$, which generalizes the known result for $q=2$. The general duality is a useful tool for studying $q$-ary functions. In the second question, we determined all submodules of $H_q(r,n)$. This resolves a long-standing question about the affine invariant subcodes of the Reed-Muller code and provides an explicit family of irreducible representations of $\text{GL}(n,\f_q)$ over $\f_q$.




\begin{thebibliography}{99}

\bibitem{Assmus-Key-1998}
E. F. Assmus, Jr. and J. D. Key, 
{\it Polynomial codes and finite geometries}, Handbook of Coding Theory, pp. 1269 -- 1343, North-Holland, Amsterdam, 1998.

\bibitem{Barthel-Livne-DMJ-1994}
L. Barthel and R. Livn\'e,  {\it Irreducible modular representations of $\text{GL}_2$ of a local field}, Duke Math. J. {\bf 75} (1994), 261 -- 292.

\bibitem{Berger-Charpin-DM-1993}
T. Berger and P. Charpin,
{\it The automorphism group of Generalized Reed-Muller codes},
Discrete Math. {\bf 117} (1993), 1 -- 17.

\bibitem{Berger-Charpin-IEEE-IT-1996}
T. Berger and P. Charpin,  {\it The permutation group of affine-invariant extended cyclic codes}, IEEE Trans. Inform. Theory {\bf 42} (1996), 2194 -- 2209.

\bibitem{Berman-DANSSSR-1956}
S. D. Berman, {\it The number of irreducible representations of a finite group over an arbitrary field}, (Russian) Dokl. Akad. Nauk SSSR (N.S.) {\bf 106} (1956), 767 -- 769.

\bibitem{Brauer-Nesbitt-AM-1941}
R. Brauer and C, Nesbitt, {\it On the modular characters of groups}, Ann. of Math. (2) {\bf 42} (1941), 556 -- 590.

\bibitem{Brier-Langevin-ITW-2003}
E. Brier and P. Langevin, {\it Classification of Boolean cubic forms of nine variables}, In: E. Biglieri and V. Tarokh (Eds.), 2003 IEEE Information Theory Workshop (ITW 2003), pp. 179 -- 182, IEEE Press, 2003.


\bibitem{Brier-Langevin-web}
E. Brier and P. Langevin, {\it Cubics in nine variables}, http://langevin.univ-tln.fr/project/cubics/

\bibitem{Carter-Lusztig-MZ-1974}
R. W. Carter and G. Lusztig, {\it On the modular representations of the general linear and symmetric groups}, Math. Z. {\bf 136} (1974), 193 -- 242.

\bibitem{Charpin-Levy-Dit-Vehel-JCTA-1994}
P. Charpin and F. Levy-Dit-Vehel,  {\it On self-dual affine-invariant codes}, J. Combin. Theory A {\bf 67} (1994), 223 -- 244.

\bibitem{Delsarte-IEEE-IT-1970}
P. Delsarte, {\it On cyclic codes that are invariant under the general linear group}, IEEE Trans. Inform. Theory {\bf 16} (1970), 760 -- 769.



\bibitem{Dougherty-Mauldin-Tiefenbruck-arXiv:2106.13910}
R. Dougherty, R. D. Mauldin, M. Tiefenbruck, {\it The covering radius of the Reed-Muller code $RM(m-4,m)$ in $RM(m-3,m)$},  IEEE Trans. Inform. Theory {\bf 68} (2022), 560 -- 571.

\bibitem{Green-LNM-1980}
J. A. Green, {\it Polynomial Representations of $\text{GL}_n$}, 
Lecture Notes in Mathematics, 830, Springer-Verlag, Berlin-New York, 1980. 

\bibitem{Hou-JCTA-1996}
X. Hou, {\it The covering radius of $R(1,7)$ --- a simpler
proof}, J. Combin. Theory A {\bf 74} (1996), 337 -- 341.

\bibitem{Hou-DM-1996}
X. Hou, {\it $\text{\rm GL}(m,2)$ acting on $R(r,m)/R(r-1,m)$},
Discrete Math. {\bf 149} (1996), 99 -- 122.

\bibitem{Hou-JCTA-2005} 
X. Hou, {\it Enumeration of certain affine invariant extended cyclic codes}, J. Combin. Theory A, {\bf 110} (2005), 71 -- 95. 

\bibitem{Hou-IJICT-2010}
X. Hou, {\it Enumeration of ${\rm AGL}(\frac m3,\,{\Bbb F}_{p^3})$-invariant extended cyclic codes},
International Journal of Information and Coding Theory, {\bf 1} (2010) 214 -- 243. 

\bibitem{James-Kerber-1981}
G. James and A. Kerber, {\it The Representation Theory of the Symmetric Group},  Encyclopedia of Mathematics and its Applications 16, Addison-Wesley Publishing Co., Reading, MA, 1981.

\bibitem{Kasami-Lin-Peterson-IC-1968}
T. Kasami, S. Lin, W. W. Peterson,  {\it Some results on cyclic codes which are invariant under the affine group and their applications}, Information and Control {\bf 11} (1968), 475 -- 496.

\bibitem{Key-McDonough-Mavron-DM-2010}
J. D. Key, T. P. McDonough, V. C. Mavron, {\it Reed-Muller codes and permutation decoding}, Discrete Math. {\bf 310} (2010), 3114 -- 3119.

\bibitem{Key-McDonough-Mavron-DM-2017}
J. D. Key, T. P. McDonough, V. C. Mavron, {\it Improved partial permutation decoding for Reed-Muller codes}, Discrete Math. {\bf 340} (2017), 722 -- 728.

\bibitem{Lang-2002}
S. Lang, {\it Algebra}, Springer, New York, 2002.

\bibitem{MacWilliams-BSTJ}
J. MacWilliams, {\it Permutation decoding of systematic codes}, Bell System Techn. J. {\bf 43} (1964), 485 -- 505.

\bibitem{Mortimer-thesis-1977}
B. Mortimer, {\it Some Problems on Permutation Groups: Affine Groups
and Modular Permutation Representations}, Ph.D. Dissertation, Westfield College,
University of London, 1977.

\bibitem{Mykkeltveit-IEEE-IT-1980}
J. J. Mykkeltveit, {\it The covering radius of the $(128,8)$ Reed-Muller code is $56$}, IEEE Trans. Inform. Theory {\bf 26} (1980), 359 -- 362.

\bibitem{Reiner-PAMS-1964}
I. Reiner, {\it On the number of irreducible modular representations of a finite group}, Proc. Amer. Math. Soc. {\bf 15} (1964), 810 -- 812.

\bibitem{Wedderburn-1964}
J. H. M. Wedderburn,  {\it Lectures on Matrices}, Dover Publications, Inc., New York, 1964.


\end{thebibliography}
\end{document}